\documentclass{birkjour}
\usepackage{amsfonts,amssymb,amsmath,mathtools,hyperref}
\usepackage[english]{babel}
\usepackage{bigints}
\usepackage{amssymb}
\usepackage{pdfsync}

\usepackage[normalem]{ulem}
\usepackage{cancel}

\numberwithin{equation}{section}

\usepackage{relsize,bigints}

\usepackage{graphicx,tikz}

\usepackage{color}

\newcommand{\eps}{\varepsilon}
\newcommand{\commentout}[1]{}

\newcommand{\la}{\lambda}

\newcommand{\cal}{\mathcal}
\newcommand{\alg}{L^1(\RR)}

\newtheorem{thm}{Theorem}[section]
\newtheorem{lemma}[thm]{Lemma}
\newtheorem{proposition}[thm]{Proposition}

\newtheorem{remark}[thm]{Remark}

\newcommand{\bes}{\begin{displaymath}}
\newcommand{\ees}{\end{displaymath}}
\newcommand{\be}{\begin{equation}}
\newcommand{\ee}{\end{equation}}
\newcommand{\ba}{\begin{eqnarray}}
\newcommand{\ea}{\end{eqnarray}}
\newcommand{\bas}{\begin{eqnarray*}}
\newcommand{\eas}{\end{eqnarray*}}

\newcommand{\B}{{\@Bbb B}}
\newcommand{\C}{{\@Bbb C}}

\newcommand{\F}{{\@Bbb F}}
\renewcommand{\P}{{\@Bbb P}}

\newcommand{\Q}{{\@Bbb Q}}
\newcommand{\bQ}{{\@Bbb Q}}
\newcommand{\N}{{\@Bbb N}}
\newcommand{\R}{{\@Bbb R}}

\newcommand{\W}{{\@Bbb W}}

\newcommand{\cA}{\@s A}
\newcommand{\cB}{\@s B}
\newcommand{\cC}{\@s C}
\newcommand{\cD}{\@s D}
\newcommand{\cE}{\@s E}
\newcommand{\cF}{\@s F}
\newcommand{\cG}{\@s G}
\newcommand{\cH}{\@s H}
\newcommand{\cI}{\@s I}
\newcommand{\cJ}{\@s J}
\newcommand{\cK}{\@s K}
\newcommand{\cL}{\@s L}
\newcommand{\cN}{\@s N}
\newcommand{\cM}{\@s M}
\newcommand{\cO}{\@s O}
\newcommand{\cP}{\@s P}
\newcommand{\cR}{\@s R}
\newcommand{\cS}{\@s S}
\newcommand{\cT}{\@s T}
\newcommand{\cV}{\@s V}
\newcommand{\cW}{\@s W}
\newcommand{\cX}{\@s X}
\newcommand{\cY}{\@s Y}
\newcommand{\cZ}{\@s Z}

\newcommand{\bma}{\@bm a}
\newcommand{\bmb}{\@bm b}
\newcommand{\bmc}{\@bm c}
\newcommand{\bmd}{\@bm d}

\newcommand{\bme}{\@bm e}
\newcommand{\bmf}{\@bm f}
\newcommand{\bmg}{\@bm g}
\newcommand{\bmh}{\@bm h}
\newcommand{\bmi}{\@bm i}
\newcommand{\bmj}{\@bm j}
\newcommand{\bmk}{\@bm k}
\newcommand{\bml}{\@bm l}
\newcommand{\bmm}{\@bm m}
\newcommand{\bmn}{\@bm n}
\newcommand{\bmo}{\@bm o}
\newcommand{\bmp}{\@bm p}
\newcommand{\bmq}{\@bm q}
\newcommand{\bmr}{\@bm r}
\newcommand{\bms}{\@bm s}
\newcommand{\bmt}{\@bm t}
\newcommand{\bmu}{\@bm u}
\newcommand{\bmw}{\@bm w}
\newcommand{\bmv}{\@bm v}
\newcommand{\bmx}{\@bm x}
\newcommand{\bx}{\@bm x}
\newcommand{\bmy}{\@bm y}
\newcommand{\bmz}{\@bm z}

\newcommand{\by}{\@bm y}
\newcommand{\bmzero}{\@bm 0}

\newcommand{\gA}{\@g A}
\newcommand{\gD}{\@g D}
\newcommand{\gJ}{\@g J}
\newcommand{\gF}{\@g F}
\newcommand{\gM}{\@g M}
\newcommand{\gR}{\@g R}

\renewcommand{\labelenumi}{$(\roman{enumi})$}

\newcommand{\e}{\mathrm e}
\newcommand{\ud }{\, \mathrm d}
\newcommand{\ude }{\mathrm d}
\newcommand{\sem}[1]{\{\e^{t#1}, t \ge 0\}}
\newcommand{\dom}[1]{\mathcal D(#1)}
\newcommand{\grae}{\lim_{\eps \to 0+}}
\newcommand{\lam}{\lambda}
\newcommand{\rla}{R_\lam}

\newcommand{\grat}{\lim_{t\to 0+}}
\newcommand{\grato}{\lim_{t\to \infty}}

\newcommand{\RR}{\mathbb R}
\newcommand{\rez}[1]{\left (\lam - #1\right )^{-1}}

\newcommand{\mc}{\mathcal}

\newcommand{\dab}{\Delta_{p,q}}

\newcommand{\wve}{\widetilde \varphi_\eps}
\newcommand{\ve}{\varphi_\eps}

\title[Diffusion approximation for a   kinetic model]
{Diffusion approximation for a simple kinetic model with asymmetric interface}

\begin{document}

\author{Adam Bobrowski}
\email{a.bobrowski@pollub.pl}
\address{Lublin University of Technology, Nadbystrzycka 38A, 20-618 Lublin, Poland.}

\author{Tomasz Komorowski}
\email{tkomorowski@impan.pl}
\address{Institute of Mathematics, Polish Academy of Sciences,
  \'{S}niadeckich 8, 00-656 Warsaw, Poland.}

\date{\today {\bf File: {\jobname}.tex.}} 

\thanks{A.~Bobrowski was supported by the National Science Centre (Poland) grant
   2017/25/B/ST1/01804. T. Komorowski  acknowledges the
support of the  Polish  National Science Centre:
Grant No.  2020/37/B/ST1/00426.}

\begin{abstract}We study a diffusion approximation for a
 model of stochastic motion {of a
particle} in one spatial dimension. The velocity of the particle is
constant but the direction of the motion
undergoes random changes with a Poisson clock. Moreover, the particle
interacts with an interface in such a way that it can
randomly be reflected, transmitted, or killed, and the corresponding probabilities depend on
whether the particle arrives at the interface from the left, or
right.  We prove that the limit process is a {\em minimal Brownian
  motion},  if the probability of killing is positive. In the case
of no killing, the limit is  a {\em  skew Brownian motion}. \end{abstract}

\thanks{Version of \today}


\maketitle

\vspace{-0cm}
\section{Introduction}
 
The paper is devoted to  the following model of stochastic motion {of a
particle} on
two copies of real line (see Figure \ref{rys1}). When on the upper
copy, denoted by $\RR \times \{1\}$, {the particle moves}
deterministically to the right with a constant velocity, which we
normalize to be equal to one; when on the lower copy $\RR \times
\{-1\}$ {it moves} to the left with the same normalized
velocity. {At the points {$(0,\pm1)$} there is an interface, which
randomly perturbs the deterministic motion and enables the particle
to switch between the upper and lower copies of the real line. It is
described by four non-negative parameters $p,p',q$ and $q'$ such that both
$p+p'\le 1$ and $q+q'\le 1$. A~particle approaching the interface from
the left, thus {moving} on the upper copy, filters through the
interface with probability $p$ and continues its motion to the right on $\RR \times
\{1\}$. With probability $p'$, the particle is
reflected and starts moving to the left (from $(0,-1)$) on the lower
copy. Finally, with a possibly nonzero probability $p_0 \coloneqq 1 - p
- p'$, the particle is killed and removed from the state space.} Analogously, when approaching the interface from the right (on the lower copy), the particle filters through the interface with probability $q$, is reflected and continues its motion on the upper copy with probability $q'$, or is killed and removed from the state-space with probability $q_0 \coloneqq 1 - q - q'$. 

Additionally, {we introduce the following scattering mechanism that
allows the particle to randomly change its direction (thus switching
between the copies of the lines) in the bulk (i.e., outside the
interface):} at the epochs of a Poisson process, {independent of
the random mechanism at the interface}, the particle moving
to  the right changes its direction   to the left, and vice versa, by
jumping from one copy of the line to the other. 

{The {entire} random mechanism
  described above   is somewhat
related to the kinetic model of a motion of a phonon with an interface,  studied in
\cite{tomekkinetic,tomekkinetic3,tomekkinetic2,tomekkinetic1}}, and the telegraph
process   with elastic boundary at the origin \cite{wlosi,wlosi1}.

After presenting, in Section \ref{aits}, the semigroups  that are involved in the model, we formulate two theorems on  diffusion approximation. The first, Theorem \ref{thm:2},  concerns the case where 
there is no killing at the interface, that is,  
 \begin{equation}
\label{zal1} 
p+p'=q+q' =1. 
\end{equation} 
The result says that, when diffusively
scaled, the density  of  particles undergoing the motion described by
the model   is well approximated by {the
density of population of
particles on $\RR$ performing independent Brownian motions} with a trace of semi-permeable membrane at
$x=0$. The above means that in the space of absolutely integrable
  functions the domain of the generator $\frac12\Delta_{p,q}$ of such Brownian motions  
consists of functions $\phi$ satisfying  the  transmission conditions,
\begin{equation}
\label{int:1} 
 p\phi (0-) = q \phi (0+) \qquad \mbox{and}\qquad \phi '(0+) = \phi '(0-) 
\end{equation}
and we have $\Delta_{p,q}\phi= \phi''$ for such $\phi$.
For $p=q$ this is the generator of the standard Brownian motion; for $p\not =q$ the related process is known as skew Brownian motion  --- see \cite[p. 45 eq. (57)]{lejayskew}, \cite[p. 107]{yor97}  and \cite[pp. 115-117]{manyor}. 
Analytic properties of this processes have been recently studied
in \cite{tombatty}, see also \cite[Chapters 4 and
11]{knigazcup}.  A transmission condition that is  analogous to
\eqref{int:1} appears, for a
nonlinear parabolic equation,  in the hydrodynamic limit of a symmetric simple
exclusion process, as viewed from  a tagged particle moving under  the action of an
external constant driving force, see
  \cite[Equation (1.4)]{landim}. The proof of Theorem \ref{thm:2} is given in Section \ref{sec4}. 


Section \ref{tpf},  see Proposition \ref{prop:2} and Theorem \ref{prop:3}, 
is devoted to  the  cosine families  generated by $\dab$ and its
dual. As an application of generation theorems presented there we
provide a more detailed characterization of the process of the
Brownian motion with transmission  condition \eqref{int:1}. In particular,  in Section
\ref{sec6.3},  
closed expressions for transition probability densities of such
  processes are derived. They  reveal the asymmetric nature of the
apparently completely permeable membrane  described in the foregoing (see also Section \ref{alt1}).  


{In the second main theorem of our paper, Theorem \ref{thm:3}, we study the kinetic model in} the case in
which the killing is `effective', that is, condition $pq_0 + qp_0 +
p_0 q_0\not=0$ holds. {This scenario includes the situation in which} both
probabilities $p_0$ and $q_0$ are strictly positive. In this case, the 
processes involved,  if diffusively
scaled, are well-approximated by the minimal Brownian motion
(i.e., the Brownian motion killed at the interface), see also Remark
\ref{rmk5.2} for more details.

We note  that the case of no interface
has been extensively studied and  is well-described in the literature,
see for example \cite[Chapter 12]{ethier}. {It is intriguingly related
 to random evolutions of Griego and Hersh
\cite{gh1,gh2,pinskyrandom}, the telegraph equation and the
corresponding stochastic process  \cite{goldsteins,kac}, see also \cite{deg}. The latter is
referred throughout the literature either as the
telegraph  or   Poisson--Kac process \cite{kniga, pinskyrandom}, {and} turns out to be a
process with independent increments in the non-commutative
Kisy\'{n}ski  group, see \cite{kniga,kkac}. 
In general, the presence of
an interface    poses a technical challenge, due to the
fact that   
the domain of the generator of the semigroup describing the
 `free'  motion of a particle has to be non-trivially modified to incorporate the {transmission} condition {that describes the interface. As exemplified in this paper, the presence of the interface leads at the same time to new stochastic phenomena.} 

  Throughout the paper we adopt the following
  notation $\RR_*:=\RR\setminus\{0\}$,
  $\RR_+:= (0,+\infty)$ and $\RR_-:= (-\infty,0)$. Moreover, we use Iverson's notation: the Iverson
bracket $[P]$ equals $1$, or $0$ iff  $P$ is true, or false, respectively.

\begin{figure}
\begin{tikzpicture}[scale=0.7]
\draw [thick,->] (-5,0) -- (5,0);
\node [above] at  (4.8,0.1) {\tiny{{$\mathbb R\times\{1\}$}}}; 
\draw [thick,<-] (-5,-2) -- (5,-2);
\node [above] at  (4.8,-1.9) {\tiny{{$\mathbb R\times\{-1\}$}}}; 
\draw [thin] (0,-3)--(0,1); 
\draw [blue,->] (-2,0.1) -- (0,0.1);
\draw [blue,dashed,->] (0,0.1) -- (2,0.1);
\draw [blue,dashed,<-] (-2,-1.9) -- (0,-1.9);
\node [above] at  (1,0.2) {\tiny{Prob=$p$}}; 
\node [above] at  (-1,-1.9) {\tiny{Prob=$p'$}}; 
\draw [violet,dotted,<-] (-2,-2.1) -- (0,-2.1);
\draw [violet,<-] (0,-2.1) -- (2,-2.1);
\draw [violet,dotted,->] (0,-0.1) -- (2,-0.1);
\node [below] at  (1,-0.2) {\tiny{Prob=$q'$}}; 
\node [below] at  (-1,-2.1) {\tiny{Prob=$q$}};  
\draw [magenta,<->] (3,0)--(3,-2); 
\draw [magenta,<->] (-3,0)--(-3,-2); 
\end{tikzpicture}
\caption{\footnotesize{Random movements on two copies of real line with interface. Particles move to the left (on the lower line) and to the right (on the upper line), but may be reflected from or killed at the interface. Additionally, at random times, particles  jump from the upper to the lower line and vice versa.}}
\label{rys1}
\end{figure}
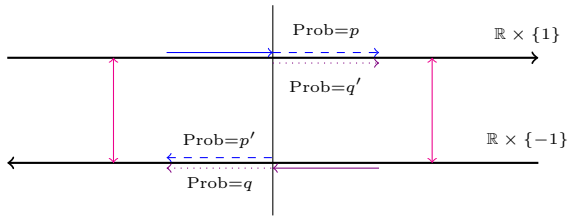

\vspace{-0.5cm}
\section{{Semigroups featured in the model}}\label{aits}


We work in the space  $L^1(S)$ of integrable functions on the {Borel}
measurable space $S\coloneqq \RR \times \{-1,1\}$ in which each copy
of {the real line} is equipped with the Lebesgue measure ${\rm m}$.
\subsection{{The semigroup corresponding to the particle motion with no scattering
  in the bulk}} 

\subsubsection{{Definition of the semigroup}}
For each $t\ge 0$, we consider {an} operator in $L^1(S)$  given by  
\begin{align*}  T(t)\varphi (x,1) & = \varphi (x-t,1) [x<0] +  \widetilde \varphi (x-t,1) [x>0] \\
T(t)\varphi (x,-1) &= \varphi (x+t,-1) [x>0] +   \widetilde \varphi (x+t,-1) [x<0], \text{ for } \varphi \in L^1(S), 
\end{align*}
where, for $ x\in \RR_*$, 
\begin{align*} \widetilde \varphi (x,1) &:= \varphi (x,1)[x>0] + (p\varphi (x,1) + q'\varphi (-x,-1) )[x<0], \\
 \widetilde \varphi (x,-1) &:=\varphi (x,-1)[x<0] + (q\varphi (x,1) + p'\varphi (-x,1) )[x>0]. 
 \end{align*} 
The family $\{T(t), t \ge 0\}$ describes the evolution of the density of particles that move according to the rules presented in Introduction, provided that switching between the
copies of the real line outside the interface is not yet allowed. The above
means  that  a given particle moves either to the right or to the left in
the bulk and its  deterministic motion is perturbed only at the
boundary.  If $\varphi \in L^1(S)$ is the density of the
population of such particles at time $0$, then $T(t)\varphi$ is its density
at time $t$.
It is clear from this description, and a direct calculation confirms
this, that $T(t)\varphi \ge 0$  and $\int_S T(t) \varphi {\ud \textrm{m}}\le
\int_S \varphi {\ud \textrm{m}}$, provided $\varphi \ge 0$, with equality
holding   when $p+p'=1$ and $q+q'=1$.  {This shows that each $T(t)$ is
a sub-Markov, or a Markov operator, respectively, in the sense of
\cite{lasota},   with the norm   $\|T(t)\| \le 1$ for all $t\ge 0$}.

Moreover, a standard {argument} shows that $\grat \|T(t)\varphi -
\varphi\|_{L^1(S)}=0$, for all $\varphi$. 
A straightforward if somewhat tedious calculation shows furthermore
that $T(t)T(s)=T(t+s)$ {(this calculation is rather unrewarding; it is
a wiser strategy to see the semigroup property as an immediate
consequence of the obvious Markovian nature of the underlying process
in the bulk, that is in $\RR_* \times\{\pm1\}$)}. The family $\{T(t), t \ge 0\}$ is therefore a strongly continuous semigroup of operators in $L^1(S)$.

\subsubsection{{The generator of $\{T(t), t \ge 0\}$}}
Turning to the task of describing the generator of this semigroup, we let the operator $A$ in $L^1(S)$ be defined as follows. Its domain $\dom{A}$ is made of $\varphi \in L^1(S)$ of the form 
\begin{align} 
\nonumber \varphi (x,1) &= (C_1 + \int_x^0 \psi (y,1) \ud y ) [x<0] + (C_2 - \int_0^x \psi (y,1) \ud y ) [x>0],\\
\label{doti:4} \varphi (x,-1) &= (C_3 - \int_x^0 \psi (y,-1) \ud y ) [x<0] + (C_4 + \int_0^x \psi (y,-1) \ud y ) [x>0],\end{align}
where $\psi \in L^1(S)$, and  the constants
$C_1,C_2,C_3$, and $C_4$ {satisfy}
\begin{equation}
\label{010408-21} 
C_2 = p C_1 + q' C_4 \quad\mbox{ and }\quad C_3 = p' C_1 + qC_4 .
\end{equation}
 In other words, a $\varphi \in \dom{A}$  is absolutely continuous in
 each of the sets $\RR_+\times \{i\} $ and $\RR_- \times
 \{i\}, i \in \{-1,1\}$ separately, with absolutely integrable
 derivatives there, and possesses finite both right and left limits at
 the points $(0,\pm 1)$   that satisfy
\begin{align} 
\varphi (0+,1) & = p \varphi (0-,1) + q' \varphi (0+,-1), \nonumber \\
\varphi (0-,-1) & = p' \varphi (0-,1) + q \varphi
                  (0+,-1). \label{doti:5} 
\end{align}
For such $\varphi, $ we define 
$$
 A \varphi (x,i) := -i \varphi '(x,i),\quad (x,i) \in \RR_*\times \{-1,1\}
  .
 $$

\begin{proposition}\label{prop:1} The operator $A$ defined above is the generator of the semigroup $\{T(t), t \ge 0\}$. \end{proposition}
\begin{proof}We check first that for any $\varphi \in \dom{A}$ the
  limit $\grat t^{-1} (T(t)\varphi - \varphi ) $ exists and equals
  $A\varphi$. To this end, we consider a $\varphi$ of the form
  \eqref{doti:4} and note that 
$$
 T(t)\varphi (x,1) - \varphi (x,1) = \int_{x-t}^t \psi (y,1)\ud y,
$$ 
provided that either $x<0$, or $x >t$. {By a direct
calculation,  similar to that presented in \cite[p. 10]{nagel} or
\cite[pp. 56--58]{ujacka}, one  can} show that 
\begin{equation}\label{pom1} \int_{\RR \setminus (0,t)} \Big|{ \frac 1t} [T(t)\varphi (x,1) - \varphi (x,1)] - \psi (x,1)\Big| \ud x \underset {t\to 0+} \longrightarrow 0.\end{equation}
Next, since for $x\in (0,t)$, we have
\begin{align*}
& T(t)\varphi (x,1) - \varphi (x,1) \\
&= p\Big[C_1 + \int_{x-t}^0 \psi (y,1) \ud y \Big] + q'\Big[C_4 + \int_0^{t-x}
  \psi (y,1)\ud y\Big] - C_2 + \int_0^x  \psi(y,1)\ud y 
\end{align*}
and the constants cancel out by the first   condition in
\eqref{010408-21},  {we show that} the following three integrals
\begin{align*}
 &
I_1(t):=\frac 1t \int_0^t \left |  \int_{x-t}^0 \psi (y,1) \ud y
   \right | \mathrm{d} x, \\
&
 I_2(t):= \frac 1t \int_0^t  \left |   \int_0^{t-x} \psi (y,1)\ud y
  \right | \mathrm{d} x,  \\
&
 I_3(t):=\frac 1t \int_0^t  \left |   \int_0^x  \psi(y,1)\ud y
  \right | \mathrm{d} x, 
\end{align*}
converge to zero, as $t\to
0+$. Fortunately, all of them are of similar type, and the
argument is pretty much the same in each case. For example, using
Fubini's Theorem, we can write
\begin{align*}
I_1(t) & \le \frac 1t \int_{-t}^0 \int_0^{y+t} \ud x   |\psi (y,1)| \ud
   y\\
&
   \le    \frac 1t \int_{-t}^0 |y+t|\, |\psi (x,1)| \ud y \le
   \int_{-t}^0  |\psi (x,1)| \ud y \underset {t\to 0+} \to 0.
\end{align*}
Convergence of {each $I_j(t)$ to $0$} allows to conclude that
condition \eqref{pom1} holds also when $\RR \setminus (0,t)$ is
replaced by the entire real line. Since the proof of convergence of
$\textstyle {\frac 1t}[T(t)\varphi (\cdot, -1) - \varphi (\cdot, -1) ]
$ to $\psi (\cdot, -1)$ on the other copy of the real line is similar,
we omit {it}.

This establishes the fact that the generator of  $\{T(t), t \ge 0\}$
is an extension of $A$. {We  conclude the proof by showing} that, in fact,  it coincides with $A$.
Since for $\lam >0$ the resolvent equation for the generator has
precisely one solution, it suffices to
  demonstrate that for any $\psi \in L^1(S)$
\begin{equation}
\label{C}
\mbox{there exists a $\varphi \in \dom{A}$ such that $\lam \varphi -
 A\varphi = \psi. $} 
\end{equation}
To this end, given $\psi \in L^1(S)$  {we} define
\begin{equation}
\label{012308-21}
\varphi (x,-1) := \e^{\lam x} \int_x^\infty \e^{-\lam y} \psi (y,-1)
\ud y\quad\mbox{ for }x >0.
\end{equation}
{It is straightforward to verify that 
$$
\lam \int_0^\infty |\varphi (x,-1)| \ud x \le  \int_0^\infty |\psi
(x,-1)| \ud x .
$$
Furthermore,  {computing $\lam \int_0^x \varphi (y,-1)\ud y$ from \eqref{012308-21}}, we
obtain that:
\begin{equation}\label{g1} \varphi (x,-1) = C_4 + \int_0^x [\lam \varphi (y,-1) - \psi (y,-1)]\ud y, \qquad x >0, \end{equation}
with
$
C_4 \coloneqq \int_0^\infty \e^{-\lam y} \psi (y,-1) \ud y.
$}
 Analogously, the function $\RR_- \ni x \mapsto \varphi (x , 1)$
 defined by  
$$
 \varphi (x,1) \coloneqq \e^{-\lam x} \int_{-\infty}^x  {\e^{\lam y}} \psi (y,1) \ud
 y
$$
 is absolutely integrable, and satisfies  
\begin{equation}\label{g2} \varphi (x,1) = C_1 + \int_x^0 [\lam \varphi (y,1) - \psi (y,1)]\ud y, \qquad x <0, \end{equation}
with 
$
C_1 \coloneqq \int_{-\infty}^0 \e^{\lam y} \psi( y, 1) \ud y.
$
 Finally, we define
\begin{align*} 
\varphi (x,1) &:= (pC_1 +q'C_4) \e^{-\lam x} + \int_0^x \e^{-\lam (x-y)} \psi (y,1) \ud y, \qquad x >0,\\ 
\varphi (x,-1) &:= (p'C_1 +q C_4) \e^{\lam x} \phantom{-} + \int_x^0 \e^{\lam (x-y)} \psi (y,-1) \ud y, \qquad x <0, 
\end{align*}
and
verify that these functions are absolutely integrable on $\RR_+$ and
$\RR_-$, respectively. Calculating as  {above} we establish that 
\begin{align*}
\varphi (x,1) &= (pC_1 +q'C_4) - \int_0^x [\lam \phi (y) - \psi (y) ]\ud y, \qquad x >0, \\
\varphi (x,-1) &= (p'C_1 +qC_4) - \int_x^0 [\lam \phi (y) - \psi (y) ]\ud y, \qquad x <0. 
\end{align*} 
These  relations, together with \eqref{g1} and \eqref{g2}, prove
that $\varphi$ satisfies condition \eqref{C}, which concludes the proof
of the proposition. 
  \end{proof}

 \subsection{{The semigroup with scattering}} 
{It is our next goal to  study} the evolution of the density of particles when the
possibility of random scattering (jumps between the copies of the real
lines) in the bulk is allowed. This evolution shall be described by a semigroup
corresponding to a generator obtained by a bounded perturbation of the
generator 
$A$.  For that purpose  {we} define a bounded linear operator} $B- I$, where $I$ is the identity operator in $L^1(S)$, whereas
\[ 
B\varphi (x,i) = \varphi (x,-i), \qquad (x,i)\in
\RR\times\{-1,1\},\,\varphi \in L^1(S).
\] 
{More precisely, if  particles perform only jumps between the  copies
of the real lines   
at the epochs of a Poisson process with intensity $\lam$, then the
dynamics of its law   is governed by the semigroup of Markov
operators generated by $\lam (B - I)$. Since $B$ is bounded,  the
Phillips Perturbation Theorem asserts that
the operator   
\begin{equation} 
G \coloneqq   A +  \lam (B - I),  \label{jeszcze}  
\end{equation}
is a generator of a strongly continuous semigroup $\sem{G}$ on $L^1(S)$. Since both $T(t)$ and $B$ are sub-Markovian
so is also each $\e^{tG}$ (see e.g. \cite[pp. 236--239]{lasota}).   {This} semigroup describes the evolution of
the density  of  particles undergoing the random
scattering (switching between $\RR\times\{\pm1\}$), both  at the
interface and in the bulk,  {as} described in the foregoing.

\subsection{{Diffusive scaling of the model}} 
Finally, to complete the setup, we consider the asymptotics of
the population density under the macroscopic scaling of both temporal and
spatial variables.  {To this end, we set $\lam =1$ in \eqref{jeszcze}
  and} introduce the diffusive {space-time} scaling $t'=\eps^2
t$ and $x'=\eps x$, where $\eps>0$ is a small parameter that will
eventually tend to $0$. Here $(t',x')$ and $(t,x)$ represent the 
macro- and microscopic variables. The deterministic velocity is
not scaled and remains equal to $1$.  The evolution of the particle
density profile in the macroscopic variables is then governed by  the generator  
\begin{equation} G_\eps \coloneqq \textstyle{ \frac 1\eps}  A + {  \frac 1{\eps^2}} (B - I), \label{dif1} \end{equation}
defined on $\dom{G}\coloneqq \dom{A}$, with the respective semigroup
$\{\e^{tG_\eps}, t \ge 0\}$. In what follows we  shall investigate the
asymptotics of $\{\e^{tG_\eps}, t \ge 0\}$, as $\eps \to 0$.


\section{A limit theorem: no loss of probability mass at the interface}\label{alt1} 

Our main theorem of this section says that in the case where condition \eqref{zal1} is satisfied,
 the scaling of \eqref{dif1} leads in the limit, as $\eps\to0+$, to a diffusion on a real line with a trace of semi-permeable membrane at $x=0$.

Suppose that at least one of the two numbers $p$ and $q$ is non-zero. We define the domain of 
an operator  $\dab$ in $L^1(\RR)$ to {consist} of $\phi $ of the form 
\begin{align} 
\phi (x) &= \Big(qC +Dx - \int_x^0 (x-y) \psi (y)\ud y\Big)[x<0]\label{alt2:1}
  \\ 
&\phantom{=}+  \Big(pC +Dx  {+ \int_0^x (x-y) \psi (y)\ud y}\Big)[x>0], \qquad x\in \RR_*,\nonumber \end{align}
where $C$ and $D$ are constants and $\psi$ belongs to $\alg $. {For any
such $\phi$ we let}  
\[
 \dab \phi = \psi := \phi''. 
\]  
In other words, $\phi \in \dom{\dab}$ is continuously differentiable
on each of the two half-lines {$\RR_\pm$} (separately) and $\phi'$
is absolutely continuous with absolutely integrable
$\phi''$. Moreover, the right and left limits of both $\phi$ and $\phi'$ at $x=0$ exist, are finite and  related by condition \eqref{int:1}. 

 As proved in \cite{tombatty}, see also \cite[Chapters 4 and 11]{knigazcup}, $\dab$ is the generator of a strongly continuous semigroup of Markov operators in $\alg$. 
This {operator} plays a crucial role in what follows. More
precisely, its isomorphic image will be shown to govern the diffusion
approximation to our model.  

To bring the operator $\dab$ closer to
  a reader we note first that in the symmetric case when  $p=q$, it
reduces to {the generator of the standard heat semigroup} 
in $L^1(\RR)$, see e.g. \cite[pp. 232-234]{lasota}. This agrees with
the intuition that in the limit the interface is completely invisible,
being totally permeable. However, the asymmetric case is remarkably
different. In particular, in contrast to the standard Brownian motion,
transition probabilities for the Markov process governed by $\frac 12
\dab$ are skewed either to the left or to the right, depending on
whether $p$ is smaller or larger than $q$. The related stochastic
  process has been known in the literature since 1970s (see
  \cite{ito,walsh}) as a \emph{skew Brownian motion}; it behaves like a Brownian motion except that the sign of each excursion is chosen using an independent Bernoulli random variable -- see the already cited survey article of A. Lejay \cite{lejayskew} for more information.

Results presented in \cite{tombatty} and \cite[Chapters 4 and 11]{knigazcup},
allow interpreting the skew Brownian motion differently, as a Brownian motion with  a trace of
semi-permeable membrane at $x=0$. Namely, the process {can} be
obtained as a limit of \emph{snapping out} Brownian motions  in which
$x=0$ is a semi-permeable membrane characterized by two permeability
coefficients: one for filtering from the left to the right and
another one for filtering from the right to the left,  equal to, say,
$\lam_p$ and $\lam_q$, respectively {(see \cite[Chapter 11]{knigazcup} and \cite{lejayn} for more on such Brownian motions)}. The
process related to $\frac 12 \dab$ is obtained in the limit, as both
$\la_p$ and $\la_q$ tend to infinity in such a way that $\lam_p/\lam_q=p/q$. Therefore, although $x=0$ seems to be totally permeable, there is
still  asymmetry between the way particles filter from the right to
the left and in the opposite direction. We refer to the works
cited above for more details; more information will also be
provided in Section \ref{tpf}.

Before stating the main result of this section, we  make two
  additional observations. First, obviously, the space 
$L^1(\RR)$ is isometrically isomorphic to the subspace $L_0$ of
$L^1(S)$ made of functions $\varphi $ such that $\varphi (x,1)=\varphi
(x,-1), x \in \RR$. The isomorphism we have in mind is $J:L^1(\RR)\to L_0$ given
by 
\begin{equation}
J\phi (x,i) = \textstyle{\frac 12}\phi (x), \quad (x,i)\in \RR\times\{-1,1\}, \phi \in L^1(\RR),\label{jot}
\end{equation}
{with $J^{-1}\varphi (x) = 2\varphi (x,1), x\in \RR, \varphi \in L_0$.} 
 It follows that the operators
\[ S(t) \coloneqq J \e^{t\frac 12 \dab} J^{-1},\qquad  t \ge 0 \]
form a strongly continuous semigroup of operators in $L_0$. Its
generator is 
\begin{equation}
\label{tdab}\textstyle{\frac 12} \widetilde \dab \coloneqq \textstyle{\frac 12} J\dab
J^{-1},
\end{equation}
 with the domain equal to  
$\dom{\widetilde \dab} =
J\Big(\dom{\dab}\Big)$.
That is, a $\varphi \in L_0$ belongs to
$\dom{\widetilde \dab} $ iff $J^{-1} \varphi $ is in $\dom{\dab}$
and then  $\widetilde \dab \varphi = J\dab J^{-1} \varphi  $, see e.g. \cite[Section 7.4.22]{kniga}.  

Secondly, observe that $B^2 = I$. Therefore, 
\[ \e^{t(B - I)}= (\e^{-t} \sinh t)  \, B + (\e^{-t} \cosh t )\, I \]   
implying that the strong limit  
\begin{equation}
\label{alt:1}  \grato \e^{t(B - I)} \eqqcolon P 
\end{equation}
exists and equals $\frac 12 (B+I).$ We note also that
\[ P\varphi (x,i)= \frac 12 (\varphi (x,1) + \varphi (x,-1)), \qquad (x,i)\in \RR\times \{-1,1\}.\]
Hence,
 $P: L^1(S)\to L^1(S)$ is a projection of $L^1$ onto $L_0$, which preserves the norm of non-negative elements of $L^1(S)$.

\begin{thm}\label{thm:2}  
Assume that condition \eqref{zal1} is satisfied  and  $p+q>0$. Let
$\widetilde \dab$ be the isomorphic image of $\dab$ in $L_0$, {defined
in \eqref{tdab}}.
Then
\[ 
{\grae} \e^{tG_\eps } \varphi = \e^{t\frac 12 \widetilde \dab} P \varphi, \qquad t >0, \varphi \in L^1(S), \]
{strongly in the norm of $L^1(S)$}, and the limit is uniform in
$t$ on compact subsets of $(0,\infty)$. For $\varphi \in L_0$, the
limit holds also for $t=0$ and is uniform in $t$ on compact subsets of $[0,\infty)$. 
\end{thm}
This theorem will be proved in Section \ref{sec4}.
\begin{remark}\rm
 Suppose that $\varphi\in L^1(S)$  is a probability
  density. Then, according to Theorem \ref{thm:2},  
 both $2\e^{tG_\eps } \varphi(\cdot,i)$, $i\in\{-1,1\}$   become
  asymptotically  equal to the density of a Brownian motion with a trace of semi-permeable membrane at $x=0$.
\end{remark}


\section{Proof of Theorem \ref{thm:2}} 

\label{sec4}
Let $\mc A$ be the extension of the
operator $A$ to the domain made of $\varphi$ of the form
\eqref{doti:4}, which   {need not}  satisfy \eqref{doti:5}, and given by
\begin{equation}
\label{cA}
\mc A \varphi (x,i) = -i \varphi'(x,i), \qquad (x,i)\in \RR_* \times
\{-1,1\}. 
\end{equation}
 Also, we let 
\begin{equation}
\label{cGe}
\mc G_\eps \coloneqq
\textstyle \frac {1}{\eps} \mc A + \frac 1{\eps^2} (B - I) 
\end{equation}
be  the extension of
$G_\eps $ defined in \eqref{dif1} to $\dom{\mc A}$.
 
\begin{lemma}\label{lemat2} 
For any  $\lam >0$  and $\eps>0$, the kernel of $\lam - \mc G_\eps $ is a two dimensional linear space spanned by $\varphi_-=\varphi_-(\eps, \lam) $ and $\varphi_+=\varphi_+(\eps, \lam )$ defined by:
\begin{align}
\nonumber 
\varphi_{-} (x,i) &\coloneqq \e^{\mu x} \Big\{ 1[i=1] + w_+
                    [i=-1]\Big\} [x<0 ]  \\
\varphi_{+} (x,i) &\coloneqq \e^{-\mu x} \left \{ 1[i=1] + w_- [i=-1]\right \} [x>0 ],\label{alt:2}
\end{align} 
 where 
\begin{equation}
\label{mu}
\mu \coloneqq \mu (\eps, \lam) = \sqrt{\lam (\lam \eps^2 + 2)} 
\end{equation} and 
\begin{equation}
\label{wpm1}
w_{\pm} = w_{\pm } (\eps, \lam ) \coloneqq \lam \eps^2 \pm \mu \eps +
1.
\end{equation}
\end{lemma} 
\begin{remark}\rm 
A simple direct calculation shows that 
\begin{equation}
\label{wpm}
w_+ w_- =1 .  
\end{equation}
Moreover, 
\begin{equation}
\grae w_{\pm} (\eps, \lam ) = 1 \quad \text{ and } \quad \grae \mu
(\eps,\lam) = \sqrt{2\lam}.\label{grae} 
\end{equation}
\end{remark}

\begin{proof}[Proof of Lemma \ref{lemat2}]For a $\varphi \in \dom{\mc A}$ to belong to the kernel
  of $\lam - \mc G_\eps$, {we need to have}
$$(\lam \eps^2 +1) \varphi - \eps \mc A \varphi = B\varphi.$$ 
{In other words, because of \eqref{cA},} 
\begin{align}\label{alt:3} 
(\lam \eps^2 + 1) \varphi (x,1) + \eps \varphi '(x,1) &= \varphi (x,-1), \nonumber \\ 
(\lam \eps^2 + 1) \varphi (x,-1)  -\eps \varphi '(x,-1) &= \varphi (x,1),\qquad x\in \RR_*.
\end{align}
These relations and the definition of $\dom{\mc A}$ (see \eqref{doti:4}) imply that the functions $\RR_* \ni x \to \varphi
'(x,i)$, $i \in \{-1,1\}$ are absolutely continuous in each half-line
$\RR_\pm$ separately, with $\varphi'' \in L^1(\RR)$. Thus,
substituting
$\varphi (x,-1)$ from the first equation of \eqref{alt:3}  into the
second one and recalling the definition of $\mu$ (cf \eqref{mu})) we obtain 
\[
\mu\varphi (x,1) = \varphi'' (x,1), \qquad x \in \RR_*.
\]
Solving this ordinary differential equation, under the constrain
that $\varphi\in L^1(S)$,  we conclude that
\[ 
\varphi (x,1) = C_1\e^{\mu x} [x<0 ] + C_2\e^{-\mu x} [ x>0]
\]
for some constants $C_1$ and $C_2$. 
Because of the first equation in
\eqref{alt:3}, it is immediate that $\varphi$ must be a linear
combination of $\varphi_-$ and $\varphi_+$. Furthermore, by a direct
calculation, using \eqref{wpm}, it can be verified that
both $\varphi_-$ and $\varphi_+$   belong to the kernel of $\lam - \mc G_\eps .$ \end{proof}


\begin{proof}[Proof of Theorem \ref{thm:2}] Relation \eqref{alt:1}
  allows us to work in the framework of the singular perturbation
  theorem of T. G. Kurtz (\cite{ethier,kurtzper,kurtzapp} or
  \cite{knigazcup}, Theorem 42.1). To prove  Theorem \ref{thm:2} we need
  to show  that 
\begin{enumerate}
\item for any $\varphi_0 \in \dom{\widetilde \dab}$ there are $\ve \in \dom{G_\eps} = \dom{A}$ such that 
$$
\grae \ve = \varphi_0\quad\mbox{ and }\quad \grae G_\eps \ve = {\textstyle \frac 12}
\widetilde \dab \varphi_0,
$$ 
\item 
for any $\phi \in \dom{A}$ we have
 $$
\grae \eps^2 G_\eps \phi = (B- I)\phi .
$$ 
\end{enumerate}
Since \emph{(ii)} follows {immediately from \eqref{dif1}}, we are left with showing \emph{(i)}.


Let $\varphi_0$ belong to $\dom{\widetilde \dab}$ and let \begin{equation}\phi
\coloneqq  J^{-1} \varphi_0,\label{fi} \end{equation} so that \eqref{alt2:1} holds for some
$C,D\in \RR$ and  $\psi \in L^1(\RR)$.  Since both $\phi$ and
$\phi''$ are in $\alg$, so is $\phi'$ (see e.g. \cite{kato} p. 192). Following \cite{ethier} p. 471, we  define 
\begin{align}
\label{tphie}
\wve (x,i) & \coloneqq  \frac 12 \phi (x) - \frac {\eps i}4 \phi' (x) \nonumber \\ 
&\phantom{:}= \varphi_0 (x,i) - \frac {\eps i}2  \varphi_0' (x,i), \qquad
  (x,i) \in \RR_*\times \{-1,1\}.
\end{align}
Then, all $\wve$, $\eps >0$ belong to $\dom{\mc A}\subset L^1(S)$. A
straightforward computation shows   that 
\begin{equation}
\label{010508-21}
\mc G_\eps \wve =
{\textstyle \frac 12} \widetilde \dab  \varphi_0
\end{equation}
(see \eqref{cGe} for the
definition of $\mc G_\eps $) whereas, obviously {from the
  definition}, 
\begin{equation}
\label{020508-21}
\grae \wve = \varphi_0.
\end{equation}   Thus, we {would  be done,}
were it not for the fact that in general $\wve\not \in
\dom{A}$. Therefore,  we need to modify the definition of $\wve$ in
order to obtain $\ve \in \dom{A}$ such that 
\eqref{020508-21} is satisfied and \eqref{010508-21} holds at least 
asymptotically.

The following construction uses the ideas of \cite{greiner}
(see also Theorem 3.1 in the more recent \cite{marta+g}),
\cite[pp. 230--232]{banbob2}  and \cite[Lemma 2.3]{banasiak}. Let
$F:\dom{\mc A} \to \RR^2$ be {a linear mapping}   given by (c.f. \eqref{doti:5})
\begin{equation}
\label{F} F 
\varphi = \begin{pmatrix}  \varphi (0+,1) - p \, \varphi (0-,1) - q'
  \varphi (0+,-1) \\
\\
\varphi (0-,-1) - p' \varphi (0-,1)  - q \, \varphi
(0+,-1) \end{pmatrix}, \qquad \varphi \in \dom{\mc A}.
\end{equation}
Note that  
\begin{equation}
\label{cAA}
\mbox{$\varphi \in \dom{\mc A}$ belongs to $\dom{A}$ iff
$F\varphi =0.$}
\end{equation}

 Let $\lam >0$ be fixed. By Lemma \ref{lemat2}, the kernel of $\lam -
 \mc G_\eps $ is a two-dimensional subspace of $L^1(S)$, and is
 therefore isomorphic to $\RR^2$. In other words, any member of this kernel that is of the form 
\begin{equation}\label{alt:3a} 
\varphi = C_1\varphi_- + C_2
  \varphi_+ \end{equation}
can be identified with the pair $(C_1,C_2)\in \RR^2$. 
The  lemma
implies further that for  $\varphi$ of this form  we have
\[ 
F\varphi 
 = \begin{pmatrix} -pC_1 + (1-q'w_-)C_2 \\ 
\\
(w_+ -p')C_1 - qw_-
   C_2 \end{pmatrix}. 
\]
By \eqref{zal1} and \eqref{wpm},  it follows that the determinant of the matrix $M_\eps $ representing the linear mapping \( \RR^2 \ni (C_1,C_2) \mapsto \varphi \mapsto F\varphi \in \RR^2\)
equals 
\begin{align*} 
\det (M_\eps)  =  pqw_- + (1-q'w_-) (p' - w_+)  =  (1-p-q)
                          (1-w_-) + 1-w_+ .
\end{align*}
Using \eqref{mu} and \eqref{wpm1} we conclude that
\begin{align} 
\det (M_\eps) = -2\lam \eps^2 - \eps (p+q) (\mu - \lam \eps).\label{alt:3b} 
\end{align}
{The determinant}  is strictly negative because $\mu >\eps
\lam $, {see \eqref{mu}}. Thus, we infer that for any $u=(u_1,u_2)\in \RR^2$,
there exists a unique $\varphi \eqqcolon K_{\lam, \eps} u$ in the
kernel of $\lam - \mc G_\eps$ such that 
\begin{equation}
\label{FK}
FK_{\lam, \eps} u = u, \quad u \in \RR^2
\end{equation}
 and 
it  is given by \eqref{alt:3a} with 
\begin{align*}
C_1= C_1 (\eps) & \coloneqq 
\frac {\det (M_{\eps,1})}{\det (M_\eps)} \coloneqq  
\frac {-qw_- u_1- (1-q'w_-)u_2}{\det (M_\eps)}, \\
C_2= C_2 (\eps) & \coloneqq 
\frac {\det (M_{\eps,2})}{\det (M_\eps)} \coloneqq  
\frac {-(w_+-p') u_1- pu_2}{\det (M_\eps)} .
\end{align*}
{Here, $M_{\eps,j}$, $j=1,2$ is the matrix formed by replacing the
$j$-th column of $M_{\eps}$ by the column vector formed by the
coordinates of $u$.}
This allows us to define
\begin{equation}
\label{phitphi}
\ve \coloneqq \wve - K_{\lam, \eps}F \wve, \qquad \eps >0 .
\end{equation}
Then $\ve$ belongs to $\dom{\mc A}$ as the difference of two
{elements} of this subspace. Moreover, $F\ve=0$, by \eqref{FK},
that is, $\ve \in \dom{A},$ see \eqref{cAA}. Also, {since $\phi$ of \eqref{fi} is of the form  \eqref{alt2:1}, by \eqref{tphie} we have:} 
$$ 
\wve (0+,i) = \frac12\Big(pC -\frac {\eps i}2 D\Big) \quad\mbox{ and } \quad \wve (0-,i) =  \frac12\Big(qC
-\frac {\eps i}2 D\Big).
$$
Hence,  it follows that 
 $u_\eps \coloneqq 
 F\wve $ (see \eqref{F}) is given by 
\[ 
u_\eps= \frac{C}{2} \begin{pmatrix} {p(1 -q -q')} \\
\\
 {q(1- p' - p)}\end{pmatrix} 
+   \frac {\eps D} 4   \begin{pmatrix}
 -(1-p+q')\\ \\
(1-q+p')\end{pmatrix} 
= \frac {\eps D} 4   \begin{pmatrix} -(1-p+q')
\\
\\ (1-q+p')\end{pmatrix} \]
because of assumption \eqref{zal1}. The respective $\det (M_{\eps, j})$,
$j=1,2$ satisfy therefore,  by \eqref{grae}, 
\begin{align*} 
\grae \frac {\det (M_{\eps, 1})}\eps &= \grae \frac D4 (qw_-(1-p+q') - (1-q'w_-)(1-q+p'))  \\
&=  \frac {qD}4 \big((1-p+q') - (1-q+p')\big) =0,
\end{align*}
and similarly
\begin{align*} 
\grae \frac {\det (M_{\eps, 2})}\eps =  \frac {pD}4 ((1-p+q') - (1-q+p')) =0.
\end{align*} 
Note  also that, see \eqref{alt:3b} and \eqref{mu},
$$
\grae \frac{\det(M_\eps)}\eps = - \sqrt {2\lam} (p +q) \not =0.
$$ 
Thus, $\lim_{\eps\to0+}C_j(\eps)=0$, $j=1,2$.

Recall that $\varphi_{\pm}(\eps, \lam )$ are the functions defined in \eqref{alt:2}.
By \eqref{grae},
their limits $\grae \varphi_{\pm}(\eps, \lam ) $ exist.  We conclude therefore that
\begin{equation}
\label{010109-21}
\grae K_{\lam,\eps} F \wve =0,
\end{equation} since
$
K_{\lam,\eps} F \wve = C_1(\eps)\varphi_-(\eps, \lam ) + C_2(\eps)
  \varphi_+(\eps, \lam ).
$
 In consequence, see \eqref{020508-21} and \eqref{phitphi},
$$\grae \ve =\grae \wve = \varphi_0 .$$ 
Finally, since $K_{\lam,\eps} F \wve \in \ker (\lam - \mc G_\eps)$, we
can write, thanks to \eqref{010508-21} and \eqref{010109-21}, that
\begin{align*}
\grae G_\eps \ve &=\grae \mc G_\eps (\wve -   K_{\lam,\eps} F\wve )= \grae (\mc G_\eps \wve - \lam K_{\lam,\eps} F\wve )\\
& =  \grae \mc G_\eps \wve= \widetilde \dab \varphi_0, 
\end{align*} as {required in (i)}. 
\end{proof}

\section{A limit theorem: loss of probability mass at the interface}
\newcommand{\cez}{C_0  [0,+\infty]}

In this section, we study our model in the case where condition \eqref{zal1} is violated. 

Let $\cez$ be the space of continuous functions $f$ on $[0,+\infty)$ such that $f(0)=0$ and the limit $\lim_{x\to +\infty} f(x) $ exists and is finite. The formula 
\[ S(t)f (x) = \int_0^\infty q_t(x,y) f(y) \ud
  y, \quad t> 0, f\in \cez, \]
where
\[
q_t(x,y):=\frac 1{\sqrt{2\pi t}}\left ( \e^{-\frac {(x-y)^2}{2t} } -
  \e^{-\frac {(x+y)^2}{2t} } \right ),\quad t>0,\,x,\,y\in \RR_+,
\]
defines a strongly continuous semigroup of operators in $\cez$ (see
e.g. \cite{kniga} Section 8.1.22, or \cite{feller} pp. 341 and
477). This semigroup describes the so-called minimal {(or killed)}
Brownian motion on $\RR_+$ in which a particle starting at $x>0$
initially behaves according to the rules of the standard Brownian
motion, but is killed and removed from the state space upon touching
$x=0$ for the first time.

\newcommand{\gen}{H}
The generator $H$ of $\{S(t), t\ge 0\}$ is defined as follows:
its domain $\dom{\gen}$ consists of twice continuously differentiable $f\in \cez$ such that $f'' \in \cez$, and $\gen f = \frac 12 f''.$ An explicit formula for $\rla \coloneqq \rez{\gen} $ reads (see again \cite{feller} p. 477): 
\begin{equation}
\label{rla}
\rla f (x) = \frac 1{\sqrt{2\lam}} \int_0^\infty \left (
  \e^{-\sqrt{2\lam}|x-y|} -  \e^{-\sqrt{2\lam}(x+y)}\right ) f(y) \ud
y, \end{equation}
for $\lam >0$ and  $f \in \cez$.

{It can be checked that} $L^1(\RR_+)$, treated as a subspace of {the
dual space $(\cez)^*$, is  invariant under the adjoint operator
$\rla^*$, and,  for any $\phi \in L^1(\RR_+)$, the formula for $\rla^* \phi $ is
given by 
 the same  expression as $\rla f $, with $f$ replaced by $\phi$.} This
 fact can be used to verify that $\rla^*$ (as restricted to
 $L^1(\RR_+)$) is the resolvent of the operator $\gen^*$, defined on
 the domain ${\cal D}(\gen^*)$ consisting of functions $\phi \in L^1(\RR_+)$ that are of the form 
\[ 
\phi (x) = Cx + \int_0^x (x-y) \psi (y) \ud y, \qquad x>0, 
\]
where $C\in \RR $ and $\psi \in L^1(\RR_+)$. It coincides with
  the space of functions  $\phi \in L^1(\RR_+)$ that have two
  generalized derivatives in $L^1(\RR_+)$ and satisfy  $\phi(0+)=0$. We have   
$
\gen^* \phi =\frac 12 \phi''
=\frac 12 \psi.
   $

Since $\gen^*$ is densely defined, and $\lam \rla^* , \lam >0$ are
sub-Markov operators, $\gen^*$ is the generator of a strongly
continuous semigroup of sub-Markov operators in $L^1(\RR_+)$, {which we
denote} $\sem{\gen^*}$ (see e.g. \cite[Corollary 7.8.1]{lasota}).  As
before, we verify  that $\e^{t\gen^*}\phi$ is formally given by the same formula as $S(t)f $ (with $f$ replaced by $\phi$). Operators $\e^{t\gen^*}, t\ge 0$ should be interpreted as follows: if $\phi$ is the initial distribution of the minimal Brownian motion, then $\e^{t\gen^*}\phi$ is its distribution at time $t\ge 0$. 

Let the domain  ${\cal D}\big(\Delta_0\big)$ of the operator $\Delta_0$ in $\alg$ be 
composed of functions $\phi $ of the form 
\begin{align} \label{alt3:1} \phi (x) &= (Cx - \int_x^0 (x-y) \psi (y)
                                        \ud y)[x<0 ] \\&\phantom{=}+
     (Dx + \int_0^x    (x-y) \psi (y) \ud y)[x>0].\nonumber
 \end{align}
Here $C$ and $D$ are real constants and $\psi $ is in $\alg$. We let
$\Delta_0 \phi := \psi (=  \phi'')$. Then, {the analysis presented above shows that} $\Delta_0$ is the generator of a
strongly continuous semigroup in $\alg$. As before, the semigroup
generated by $\frac 12 \Delta_0$ describes the minimal Brownian motion, but this time the state-space is composed of two disjoint half-axes: $\RR_- $ and $\RR_+$. In this 
process, a
particle starting at an $x\not =0$ performs a standard Brownian motion
until the first time when it touches $0$; at this moment the particle
is killed and removed from the state space. 

As in Section \ref{alt1}, in the space $L_0\subset L^1(S)$ there is an isomorphic image of the semigroup  generated by $\Delta_0$. Denoting the generator of this image by $\widetilde \Delta_0$, we obtain the following counterpart of Theorem \ref{thm:2}. Its formulation involves the probabilies
\begin{equation}
\label{p0q0}
 p_0\coloneqq  1 - p - p' \quad \text{ and } \quad q_0 \coloneqq  1 - q- q'
\end{equation}
that a particle passing through the interface is killed and removed from the state-space. 

\begin{thm}\label{thm:3} Suppose that 
\begin{equation}\label{gamma} \gamma \coloneqq  pq_0 + qp_0 + p_0 q_0\not=0. \end{equation}
Then,
\[ \grae \e^{tG_\eps } \varphi = \e^{t \frac 12 \widetilde \Delta_0} P \varphi, \qquad t >0, \,\varphi \in L^1(S), \]
(in the norm of $L^1(S)$) and the limit is uniform for $t$ in compact subsets of $(0,\infty)$. For $\varphi \in L_0$, the limit extends to $t=0$ and is uniform for $t$ in compact subsets of $[0,\infty)$. 
\end{thm}

\begin{proof} We follow the argument presented in the proof of Theorem
  \ref{thm:2} and define functions $\wve$ and mapping $F$ by formulas
  \eqref{tphie} and \eqref{F}, respectively. 
Relation \eqref{alt:3b} appearing there has been derived under
  assumption \eqref{zal1},
and therefore requires modification based on relations \eqref{p0q0}.
In general, when \eqref{zal1} need not be true, 
\begin{align*}
\det (M_\eps) &= -2\lam \eps^2 - (p+q) \eps (\mu - \lam \eps)\\
&\phantom{=}
 +
[q_0(1-p) +p_0 (1-q) - p_0q_0]w_- - p_0 - q_0. 
\end{align*}
Hence, in contrast to the case considered before, $\det (M_\eps)$ is
not of the order of $\eps$. Rather, 
$\grae \det (M_\eps) = - \gamma $, which is non-zero by assumption \eqref{gamma}. In
particular, $ \det (M_\eps)$ is non-zero  for $\eps $ small enough,
and thus the operator $K_{\lam,\eps}$ is well-defined.

Since $\phi$ is of the form \eqref{alt3:1} we have
 \[\wve (0+,i) = -\frac {\eps i} 4 D \quad\text { and  }\quad  \wve (0-,i) =
   -\frac {\eps i} 4 C, 
\] 
implying (see \eqref{F}) 
\[ 
F\wve = \frac \eps 4 \begin{pmatrix} -D +pC - q' D\\ 
\\
C+ p'C
  -qD \end{pmatrix}, 
\]
and  $\grae F\wve =0$. This in turn proves that $\grae \det (M_{\eps, i})=0$,  for
$i=1,2$ and then, as in the proof of Theorem \ref{thm:2}, that $\grae K_{\lam,\eps} F\wve =0$.  The rest of the proof runs as in the aforementioned theorem.
 \end{proof}

\begin{remark}\label{rmk5.2} \rm The assumption that $\gamma \not =0$ is natural. It
  is automatically satisfied if both $p_0$ and $q_0$ are
  positive; this agrees with our intuitions well because this is the case in which killing is possible for particles approaching the interface from both sides. Condition \eqref{gamma} is also fulfilled  in the following two cases: (a) $p_0=0$ but $q_0\not =0$ and  $p>0$, and (b) $q_0=0$ but $p_0\not =0$ and  $q>0$. 
To explain the case (a): condition $p_0=0$ combined with $p=0$ describes the scenario in which all particles approaching the interface from the left are reflected. Since such particles can never filter to the right half-axis to be possibly killed there, it is clear that the minimal Brownian motion is not a good approximation for a process with $p_0=p=0$ even if $q_0\not =0$. Interpretation of (b) is similar. 
\end{remark}

\newcommand{\cab}{C_{p,q}}
\newcommand{\fal}{\widetilde}
\newcommand{\phir}{\phi_{\mathrm {right}}}
\newcommand{\phil}{\phi_{\mathrm {left}}}
\newcommand{\psir}{\psi_{\mathrm {right}}}
\newcommand{\psil}{\psi_{\mathrm {left}}}
\section{Transition probabilities for the process governed by ${\textstyle {\frac 12}} \dab$}\label{tpf}

The operator $\dab$ of Section \ref{alt1} turns out to be not only the generator of a semigroup
but also the generator of a bounded cosine family: there is a strongly continuous family $\{\cab (t), t \in \RR\}$ of equibounded operators such that $\cab (0)=I$, $\cab (t)=\cab (-t)$ for $t>0$,  and 
\[ \lam (\lam^2 -\dab )^{-1} = \int_0^\infty \e^{-\lam t} \cab (t) \ud
  t, \qquad \lam >0.\]
We recall that, by \cite[Proposition 3.14.4]{abhn} this relation implies the cosine family functional equation:
\[ 2\cab (t) \cab (s) = \cab (s+t) + \cab (t-s), \qquad s,t \in \RR.\]
 We will
find an explicit formula for  $\{\cab (t), t \in \RR\}$  using Lord Kelvin's method of
images and, as an application, will obtain closed expressions for
transition probability {densities} for the Markov process governed by $\frac 12 \dab$.

\subsection{The cosine family generated by $\dab$}\label{tcfg} 
The basic idea of the method (see \cite{kosinusy,kelvin}) is to
represent a cosine family generated by {the} Laplace operator, with
the domain described by a boundary condition, by means of the basic cosine family 
\begin{equation}
\label{Ct}
 C(t) \phi (x) = {\textstyle \frac 12} (\phi (x+t) + \phi (x-t) ),
 \qquad x,t\in \RR, \phi \in L^1(\RR) 
\end{equation}
and of a unique extension operator that is associated with the boundary condition. As explained in \cite{tombatty}, in the case of transmission conditions the method should be appropriately modified, and in particular requires constructing two, and not just one, extension operators. Here are the details (see Figure \ref{rys2}) pertaining to the transmission conditions \eqref{int:1}.

Let $\{\cab (t), t \in \RR\} $ denote the searched for cosine family
generated by $\dab$. To find {$\cab(t) \phi $ for a given $\phi \in
L^1(\RR)$ and $t\in\RR$} we first  {discard} the part of $\phi$ on the negative half-axis, and then find a way to extend $\phi$  to the entire $\RR$ so that 
\begin{equation}
\label{Ct1}
\cab (t) \phi (x)  = C(t) \fal \phir (x), \qquad   (t,x) \in \RR\times\RR_+.
\end{equation}
Here $C(t)$ is given by \eqref{Ct} and $\fal \phir:\RR \to \RR $ is an  (yet unknown) extension of the
right part of the graph of   $\phi $. 
We stress that this formula is supposed to be valid only for $x> 0$. Its counterpart for $x<0$ is obtained similarly. We cut off the part of $\phi$ on the positive half-axis, and then extend the remaining graph to that of a function on the entire real line in such a way that 
\begin{equation}
\label{Ct2} 
\cab (t) \phi (x)  = C(t) \fal \phil (x), \qquad     (t,x) \in \RR\times\RR_-.
 \end{equation}
Here $\fal \phil $ is an (also unkown) extension of the left part of the graph of $\phi$. 
Formulas \eqref{Ct1} and \eqref{Ct2}   define then the cosine 
family $\{\cab (t), t \in \RR\} $.

\begin{figure}
 \begin{tikzpicture}
 \draw [->] (-5,0) -- (5,0);
\draw [->] (0,-1.5)-- (0,2);
\draw[blue, domain=0:5, samples=600] plot (\x, {exp(-\x)*cos(deg(3*\x))}); \draw[dotted,blue, domain=-5:0, samples=600] plot (\x, {-3*exp(\x)+ 4*exp(2*\x)}); 
\draw [blue] (-5,-0.5) -- (-3,0.75) -- (-2,0.75) -- (0,0.15);
\draw [blue, dashed] (0,0.15)-- (1,1) -- (2,0.5) -- (3,1.5)-- (4,0.5)-- (5,1.5);
\node [above] at (3,1.5) {$\fal \phil$};
\node [above] at (-2.5,0.8) {$\phi$};
\node [below] at (-1,-0.6) {$\fal \phir$};

\end{tikzpicture}
\caption{Two extensions of a single function $\phi$ (solid line): extension $\fal \phil $ (dashed) of its left part, and extension $\fal \phir$ (dotted) of its right part.}
\label{rys2}
\end{figure}

Of course, the issue lies in finding   extensions $\fal \phir$
  and $\fal \phil$. {In order to conveniently work on one half-axis,}
we  shall  find   functions $\psir, $ $\psil:\RR_+\to\RR$ {related to these extensions as follows}
\[ 
\psir (x ) = \fal \phir (-x) \quad \text { and } \quad \psil (x) = \fal \phil (x),
\qquad x >0.
\] 
Here, to our aid comes the information that a cosine family leaves the
domain of its generator invariant. Hence, if $\phi $ {belongs to}
$\dom{\dab}$, and in particular conditions \eqref{int:1} are
satisfied, then these boundary conditions are also satisfied when
$\phi$ is replaced by $\cab (t) \phi$. In terms of $\psil $ and
$\psir$ this property is expressed {in} the following system of equations: 
\begin{align*}
p[\psil (t) + \phi (-t)] &= q [\phi (t) + \psir (t) ],  \\
\phi '(t) - \psir '(t) &= \psil '(t) + \phi'(-t), \qquad t >0. 
\end{align*} 
Differentiating the first of these equations and substituting $\psil
'(t)$, calculated from the second equation, into the result, we obtain 
$$
(p+q) \psir' (t) = (p-q) \phi'(t) - 2p \phi'(-t).
$$ 
Thus
\[
 \psir (t) = \frac{p-q}{p+q} \phi (t) + \frac{2p}{p+q} \phi(-t) + c
  \qquad t >0,
 \]
where $c$ is some constant.
Since we want the extensions to be   continuous,  we also require that
$\psir (0+) = \phi (0+) $. Then, this {condition} together with the first of the {relations in} \eqref{int:1} {implies} that $c=0$. Hence,
\begin{align}
\label{psir}
\psir (t) = \frac{p-q}{p+q} \phi (t) + \frac{2p}{p+q} \phi(-t). 
\end{align}
 {Similarly,}
\begin{align}
\label{psil}
\psil (t) = \frac{2q}{p+q} \phi (t) + \frac{q-p}{p+q} \phi(-t), \qquad t >0. 
\end{align}
At this point we forget about the fact that these formulae {have been}
derived under the assumption that $\phi \in \dom{\dab}$ and make an
ansatz that for $t\ge 0$ the searched for cosine family, generated by $\dab$, is given by 
\[ 2 \cab (t) \phi (x) := \begin{cases} \phi (x-t) + \phi (x+t),  
& x\in (-\infty, -t)\cup  (t,\infty), \\
 \phi (x-t) +\psil (x+t) ,& x \in (-t,0), \\
\psir (t-x)+\phi (x+t), & x \in (0,t).
 \end{cases}\] 
By \eqref{Ct}, \eqref{psir} and \eqref{psil}, this formula  has the following equivalent form: 
\begin{equation}
  \label{Cpq0}
  \cab (t)\phi (x)  = C(t)\phi (x) \quad \text{ as long as } |x| >t,
  \end{equation}
and for the remaining $x$s we have
\begin{equation}
\label{Cpq}
  \cab (t) \phi (x)  = \begin{cases} 
 C(t) \phi (x)+ \dfrac{q-p}{2(p+q)} [\phi(-x-t)+\phi (x+t)]   
, & x \in (-t,0), \\
&\\
C(t) \phi (x) + \dfrac{p-q}{2(p+q)} [\phi (t-x) +\phi(x-t)], & x \in (0,t). 
 \end{cases}
\end{equation}
Obviously, for $p=q$ we have $\cab (t)\phi =C(t)\phi$. Formula
\eqref{Cpq} shows that the same
holds also when $\phi$ is an odd function. On the other hand, if
   $\phi$ is even  we can rewrite \eqref{Cpq} in the form
\begin{equation}
\label{Cpq1}
  \cab (t) \phi (x) = C(t) \phi (x)+ \frac{q-p}{p+q}\Big[ \phi
  (x+t)1_{(-t,0)}(x)- \phi (x-t) 1_{(0,t)}(x)\Big],\quad t>0.  
\end{equation}

 \begin{proposition}\label{prop:2} {Suppose that  $p+q>0$}. Then, $\{\cab (t), t \in \RR\}$ is a
  strongly continuous cosine family and its generator is
  $\dab$. \end{proposition}  
\begin{remark} \rm The proof follows closely the proof of  Proposition \ref{prop:1}, and hence we omit it. 
The   cosine family functional equation for $\{\cab (t), t \in \RR\}$
is a direct consequence of Proposition \ref{prop:3} presented in 
Section \ref{sec6.2} below. 
\end{remark} 

\subsection{Dual cosine family in $C[-\infty,+\infty]$} 

\label{sec6.2} 

\newcommand{\cabg}{\cab^*}
\newcommand{\dabg}{\dab^*}
\newcommand{\cer}{C[-\infty,+\infty]}

Let $\cer$ be  the space of continuous functions on $\RR$ that have
finite limits at both $+\infty$ and $-\infty$. A direct computation,
using \eqref{Cpq0} and \eqref{Cpq}, shows that the dual operators to $\cab (t) $ in $L^\infty (\RR)$ are given by 
\begin{equation}\label{jawna} \cabg (t) f (x)
  = \begin{cases}C^*(t)f(x), & |x|\ge t, \\
&\\
C^*(t)f(x)+\dfrac {q-p}{2(p+q)} [f(-x-t) - f(x+t)], & -t< x\le
0, \\
&\\
C^*(t)f(x)  +\dfrac {p-q}{2(p+q)} [f(t-x)-
f(x-t) ]  , & 0< x <  t. 
\end{cases} 
\end{equation}

\begin{remark}\rm 
Here, $C^*(t)$ is the dual to $C(t)$, given also by  formula \eqref{Ct}. 
  Interestingly, the family  $\{\cab^* (t), t \in \RR\} $  leaves the subspace $\cer$ of $L^\infty
(\RR)$ invariant. Formula
\eqref{Cpq} could be used to define a cosine family on  $L^\infty
(\RR)$, however, then  the subspace $\cer$ would not be invariant under $\{\cab (t), t
\in \RR\} $. Moreover, such a family would not be continuous in the strong
topology of $L^\infty
(\RR)$.
\end{remark} 
\begin{remark}\rm 
Interestingly, if $q>p$, the expression in the second line of \eqref{jawna} can be interpreted in probabilistic terms, but that in the third line cannot. Vice versa, if $q<p$, the third line can be interpreted probabilistically, but the second cannot. 

To wit, the second line can be equivalently written as 
\[ C(t)f (x) =  \frac 12 f(x-t) + \frac 12 \left (  \frac {2p}{p+q} f(x+t)+ \frac {q-p}{p+q} f(-x-t) \right ), \quad  -t< x\le 0 .\]    
The arguments $x-t, x+t$ and $-x-t$ used here are possible positions at time $t$ of a particle which at time $0$ starts at a point $x\in (-t,0]$ and moves according to the following rules. (a) With probability $\textstyle{\frac 12}$ it moves to the left with constant speed equal to $1$. (b) With the same probability it moves to the right with the same speed, but when it hits an interface at $0$, it either filters through the interface (with conditional probability    $\frac {2p}{p+q} $), or is reflected (with conditional probability $\frac {q-p}{p+q}$), and starts moving to the left with the same velocity as before.  

A similar interpretation of the third line is impossible, since the factor $\frac {p-q}{2(p+q)}$ appearing there  is in the case under consideration negative, and thus cannot be thought of as a probability. As we shall see in Section \ref{sec6.3}, the  entire formula for the semigroup generated by $\frac 12 \dabg$ has a natural probabilistic interpretation.  \hfill $\square$
\end{remark}

Before we prove that $\{\cabg (t), t \in \RR\}$ is a cosine family we
turn to a description of an operator that later on will be shown to be
the generator of  $\{\cabg (t), t \in \RR\}$. Namely,{ we say that an
$f\in \cer $ belongs to  $\dom{\dabg}$} if the following three
conditions are satisfied: 
\begin{itemize}
\item[(a)] $f$ is twice continuously differentiable in both
  $(-\infty,0]$ and $[0,\infty]$, separately, with left-hand and
  right-hand derivatives at $x=0$, respectively, 
\item[(b)] both the limits $\lim_{x\to \infty} f'' (x)$ and
  $\lim_{x\to -\infty} f''(x) $ exist and are finite {(it follows that, in fact, they
  have to be equal to $0$)}, and
 \item[(c)] $pf'(0+) = qf'(0-)$ and $f''(0+)=f''(0-).$  Note that this condition implies that, although $f'(0)$ need not exist, it is meaningful to speak of $f''(0)$.  
\end{itemize}
Furthermore, we
   define \(\dabg f := f'' .\)


\begin{lemma}\label{lem:2}For any $\lam >0$ and $g \in \cer$, the
  resolvent equation 
$$\lam^2 f - \dabg f = g
$$ 
has a unique solution $f\in \dom{\dabg}$ given by 
\begin{equation}\label{deff} f(x) = \begin{cases}C_+ \e^{\lam x} +D_+
    \e^{-\lam x} - \lam^{-1} \bigintss_{\,0}^x \sinh [\lam (x-y)] g(y) \ud y, &
    x \ge 0 ,\\ 
&\\
C_- \e^{-\lam x} +D_- \e^{\lam x} + \lam^{-1} \bigintss_{\,x}^0 \sinh[ \lam (x-y)] g(y) \ud y, & x \le 0 ,
\end{cases} \end{equation}
where 
\begin{equation} \label{defce} 
C_- \coloneqq \frac 1{2\lam} \int_{-\infty}^0 \e^{\lam y} g(y) \ud y \quad \text{ and } \quad C_+ \coloneqq \frac 1{2\lam} \int_0^{\infty} \e^{-\lam y} g(y) \ud y  
\end{equation}
whereas 
\begin{equation}\label{defde}
D_-\coloneqq \frac {2p}{p+q} C_+ + \frac{q-p}{p+q} C_-  \quad \text{ and } \quad D_+ \coloneqq \frac{p-q}{p+q} C_+ + \frac {2q}{p+q} C_- .\end{equation}
\end{lemma}
\begin{proof}
On the right half-axis, a solution $f$ to the resolvent equation
solves the differential equation 
\begin{equation}
\label{010409-21}
\lam^2 f(x)- f''(x) =g(x)
\end{equation} and is thus of the
form given in the upper part of \eqref{deff}. Since we {stipulate}
that the limit $\lim_{x\to \infty} f(x)$  exists and is finite, we
must take $C_+$ as in \eqref{defce}. Analogous {argument}
establishes the lower part of \eqref{deff} with $C_-$ of equation
\eqref{defce}. {Since $f$ is to belong to $\cer$, we should have $f(0-) =f(0+) $ and by condition (c) of the definition of $\dom{\dabg}$, we should have $qf'(0-) =
pf'(0+)$.} These {two} conditions hold iff 
$$
C_-+D_- = C_+ +D_+\quad\mbox{ and }\quad p(C_+ -D_+) = q(D_- - C_-)
$$
{and the unique solution to this system for unknown $D_-$ and $D_+$ is given by \eqref{defde}.}
 
{We note that,} due to the fact that $f$ satisfies the  differential
equation \eqref{010409-21} on each half axis $\RR_\pm$, the condition $f(0-) =f(0+) $ 
implies that $f''(0-)=f''(0+)$.  {This proves that the function $f$ defined above is indeed a member of $\dom{\dabg}$.} 
 \end{proof}

For a future reference, we note that our lemma immediately implies the  following alternative formula for $f= (\lam^2 - \dabg )^{-1}g$: 
 \begin{equation}\label{deffp} f(x) = \begin{cases}D_+ \e^{-\lam x} +
     \dfrac 1{2\lam} \bigintss_{\,0}^{+\infty} \e^{-\lam |x-y|} g(y) \ud y, & x \ge
     0 ,\\ 
& \\
D_- \e^{+\lam x} + \dfrac1{2\lam } \bigintss_{-\infty}^0 \e^{-\lam |x-y|} g(y) \ud y, & x \le 0 .
\end{cases} \end{equation}

\newcommand{\gier}{g_{\textrm {right}}}
\newcommand{\giel}{g_{\textrm {left}}}

\begin{thm}\label{prop:3} {Suppose that $p+q>0$}. Then,  $\{\cabg (t),
  t \in \RR\}$ is a strongly continuous cosine family with generator $\dabg$. In addition
    \begin{equation}
      \label{011309-21}
 \|\cabg (t) \| \le  M , \qquad t \in \RR,
\end{equation}
where 
\begin{equation*}
M= M(p,q) \coloneqq \frac{2\max (p,q)}{p+q}.
\end{equation*}
\end{thm} 


\begin{proof}Formula \eqref{jawna} implies that for each $g\in \cer$ and $x\in
  \RR$, the function $[0,\infty) \ni t \mapsto \cabg(t) g(x)$ is  continuous and bounded by $M \|g\|$.
A~direct calculation, using \eqref{jawna}, leads to the following
formula for the  Laplace transform: 
\[ \int_0^\infty \e^{-\lam t}\cabg (t) g(x) \ud t = \frac 12 \int_{-\infty}^\infty \e^{-\lam |x-s| } \fal \gier (s) \ud s, \]
where $\fal \gier (s) := g(s)$ for $s\ge 0$ and 
$$
\fal \gier (s) := \frac {2q}{p+q} g(s) + \frac {p-q}{p+q}
g(-s)\quad\mbox{ for }s< 0.
$$


 We note that $\fal \gier $ is a continuous function on $\RR$ that
 extends the right part of $g$.
 A glance
 at \eqref{defce} and \eqref{defde} shows that the upper part of
 \eqref{deffp} can be rewritten in the form
\[ f(x) =  \frac 1{2\lam}  \int_{-\infty}^\infty \e^{-\lam |x-y| } \fal \gier (y) \ud y, \qquad x \ge 0 .\]
Therefore, for $x\ge 0$, 
\begin{equation}\label{key} \int_0^\infty \e^{-\lam t}\cabg (t) g(x) \ud t = \lam (\lam^2 - \dabg)^{-1} g (x), \end{equation}
and a similar analysis on the left half-axis, involving the extension
$\fal \giel $ of the left part of $g$ which is given by 
$$
\fal \giel (s) = \frac {q-p}{p+q} g(-s) + \frac {2p}{p+q}
g(s)\quad\mbox{ for }s\ge 0,
$$ 
shows that \eqref{key} may be extended to the entire $\RR$.

Equation \eqref{key} is a key to the proof of the theorem. Since the mapping  $\RR_+
\ni \lam \mapsto \rez{\dabg}$ is infinitely   differentiable, even in
the operator norm, so is
$\lam \mapsto \lam (\lam^2 - \dabg)^{-1}$. Because the convergence in
the norm of $\cer $ implies the pointwise convergence, we conclude that
\[ \left [\frac {\ude^n}{\ude \lam^n} \lam (\lam^2 - \dabg)^{-1} g \right ] (x) = \int_0^\infty \e^{-\lam t} (-t)^n \cabg (t) g(x) \ud t, \]    
for $x\in \RR, \lam >0$ and $ n\in \mathbb N.$
This in turn yields  
\[ \left \| \frac {\ude^n}{\ude \lam^n} \lam (\lam^2 - \dabg)^{-1} g
  \right \|\le M\|g\| \int_0^\infty \e^{-\lam t} t^n \ud t = \frac
  {M\|g\| n!}{\lam^{{n+1}}},\]
showing that the estimates of the Sova--Da Prato--Giusti generation
theorem are satisfied, see e.g. \cite[p. 119]{goldstein}, or the
original papers  \cite{dapratogiusti} and \cite{sovac}. Since $\dabg$
is densely defined, we conclude that this operator is a generator of a
cosine family, say {$\{\mathcal C(t), t\in \RR\}$} such that $\|\mathcal
C(t) \|\le M, t \in \RR.$
Now, the relation 
$$ 
\lam (\lam^2 - \dabg)^{-1} = \int_0^\infty \e^{-\lam t} \mathcal C(t)
\ud t,
$$ 
see the already cited  \cite[Proposition 3.14.4]{abhn}, combined with the fact that the
convergence in $\cer$ implies the pointwise convergence {allows us to
infer  that }
\[ 
 \left [\lam (\lam^2 - \dabg)^{-1}g \right ](x)  = \int_0^\infty
 \e^{-\lam t} \mathcal C(t) g (x) \ud t
\]
for all $x\in \RR,\, \lam >0$ and $g\in \cer$.
But, in view of
\eqref{key}, this means that two continuous functions: $t \mapsto
\mathcal C(t)g(x)$ and $t \mapsto \cabg (t) g(x)$ (we think of $g$ and
$x$ as temporarily fixed) have identical Laplace transforms, and
consequently these functions   coincide. This proves  
that $\mathcal C(t)g= \cabg (t)g$ for all $t \ge 0$ and $g \in \cer$,
showing {in turn} that $\{\cabg (t), t \in \RR \}$ is identical to
$\{\mathcal C(t), t \in \RR\}$. In particular, the former is a
cosine family generated by $\dabg$ and estimate \eqref{011309-21} follows, as claimed. \end{proof}

\subsection{Transition probability densities for the process governed by
  ${\frac 12}\dabg $} 
\label{sec6.3}
As the generator of a cosine family, $\dabg$ is also the generator of a
semigroup. The Weierstrass formula (see e.g.  \cite[p. 219]{abhn} or \cite[p. 120]{goldstein}) tells us that 
\[ \e^{t \frac 12 \dabg} f(x) = \sqrt{ \frac 2{\pi t}} \int_0^\infty \e^{-\frac {s^2}{2t}} \cabg (s) f(x) \ud s  \]
for $ t>0, x \in \RR$ and $ f \in \cer.$ 
Here, again, we use the fact that norm convergence in $\cer$ implies pointwise convergence.   
The semigroup $\sem{\frac 12 \dabg}$ describes the same stochastic
process as $\sem{\frac 12 \dab}$ but from a different perspective:
whereas the later determines evolution of the densities, the former
determines the evolution in time of {the expected values of a family of
  observables of the  value  of the process}. More specifically, 
\begin{equation}
  \label{021309-21}
\e^{t \frac 12 \dabg} f(x) = E_x \, f(w_{p,q}(t)) 
\end{equation}
where $w_{p,q}(t), t\ge 0$ is the underlying process,   $E_x$ is the
expected value conditional on the process starting at $x$, {and $f  \in \cer$}. 

Our last goal is to combine these two relations with the explicit form
of $\cabg (t)$ given in \eqref{jawna}, to find closed form formulae
for transition {probability densities} of the process $w_{p,q}
(t), t\ge 0$, that corresponds to the generator $\Delta_{p,q}$.
 To start with,
for $x>0$, by \eqref{jawna},  we can write
\begin{align*} \sqrt{2\pi t}\,  \e^{t \frac 12 \dabg} f(x) &= \int_0^\infty \e^{-\frac {s^2}{2t} } [f(x+s) +f(x-s)] \ud s \\ &\phantom{=} +  \frac {p-q}{p+q} \int_x^\infty \e^{-\frac {s^2}{2t} } \bigl [  f(s-x) -  f(x-s)\bigr ] \ude s \\
&= \int_{-\infty}^\infty \e^{-\frac{(y-x)^2}{2t}} f(y) \ud y \\
&\phantom{=} + \frac {p-q}{p+q} \left [ \int_0^\infty \e^{-\frac {(y+x)^2}{2t}} f(y) \ud y - \int_{-\infty}^0 \e^{-\frac {(y-x)^2}{2t}} f(y) \ud y \right ].
\end{align*}
Hence, by \eqref{021309-21}, for $x > 0$ and $t\ge 0$,
\[ E_x  \, f(w_{p,q}(t)) = \int_{-\infty}^\infty \gamma_{p,q}^+ (t,x,y) f(y) \ud y , \] 
where 
\[ \gamma_{p,q}^+ (t,x,y) \coloneqq\begin{cases}  \dfrac 1{\sqrt{2\pi
        t}}  \left ( \e^{-\frac {(y-x)^2}{2t}} +  \dfrac {p-q}{p+q}
      \e^{-\frac {(y+x)^2}{2t}}\right ), & y>0, \\
&\\ \dfrac {1}{\sqrt{2\pi t}} \dfrac {2q}{p+q}  \e^{-\frac {(y-x)^2}{2t}}, &
 y<0 ,\end{cases}  
\]
is the probability density for the position of $w_{p,q} (t)$, given that the process starts at $x>0$. 
Similarly, for $x<0$, 
 \[ E_x  \, f(w_{p,q}(t)) = \int_{-\infty}^\infty \gamma_{p,q}^- (t,x,y) f(y) \ud y , \] 
 where 
 \[ \gamma_{p,q}^- (t,x,y) \coloneqq\begin{cases}  \dfrac 1{\sqrt{2\pi
         t}}  \left ( \e^{-\frac {(y-x)^2}{2t}} +  \dfrac {q-p}{p+q}
       \e^{-\frac {(y+x)^2}{2t}}\right ), & y<0, \\
&\\ \dfrac {1}{\sqrt{2 \pi t}} \dfrac{2p}{p+q}  \e^{-\frac {(y-x)^2}{2t}}, & y>0 .\end{cases}  \]
These formulae confirm the already announced fact  that,  
even though the membrane situated at $x=0$ is apparently completely
permeable for  the process governed by $\frac 12 \dabg$, there is
some  residual asymmetry between the way particles filter from the right
to the left {and in the opposite direction}. 
  
To see this, consider, for example, the case of $p>q$ in which
filtering to the right is `easier' than filtering to the left. Then,
the probability that a particle starting at an $x\not = 0$ will be in
a subset of $(0,\infty)$ at time $t$ is larger than the same
probability in the standard, that is, symmetric Brownian motion. This
is {due to the fact that}
\[  \gamma_{p,q}^- (t,x,y) > \gamma_{p,p}^- (t,x,y) \quad \text{ and } \quad \gamma_{p,q}^+ (t,x,y) > \gamma_{p,p}^+ (t,x,y), \qquad y >0 ;\]
of course, this comes at the cost of reversing these inequalities in the left half-axis. 

It is also instructive to look at the particular subcase of $q=0$
(total impermeability from the right): then $\gamma_{p,0}^+$
coincides with the transition probability density of the reflected Brownian
motion on the right half-axis, and is zero on the left half-axis. At
the same time, $\gamma_{p,0}^-$ coincides with the transition
probability density of a minimal Brownian motion on the left half-axis, and all the probability mass that is lost at $x=0$ is transferred to the right half-axis. 
\begin{remark}
\em We note that the formulae for the transition probability densities of a
  skew Brownian motion have been derived in a different way  in
  \cite[eq. (17) p. 420]{lejayskew}.  
\end{remark}

\bf Acknowledgment. \rm We would like to thank K. Burdzy for referring
us to the literature on a skew Brownian motion, and in particular 
to the paper \cite{lejayskew}
 of
A. Lejay.

\bibliographystyle{plain} 
\bibliography{bibliografia}

\end{document}